\documentclass{article} 
\usepackage[utf8]{inputenc}
\usepackage{mathtools}
\usepackage{amstext}
\usepackage{amsfonts}
\usepackage{amssymb}
\usepackage{amsthm}
\usepackage{amsmath}
\usepackage{url}
\usepackage{enumitem}
\usepackage{cases}
\usepackage{verbatim}
\usepackage{xcolor}
\usepackage[backend=biber]{biblatex}
\usepackage{thmtools}
\usepackage{colonequals}

\newcommand{\pr}[1][]{\mathbb{P}}
\newcommand{\sm}[0]{\setminus}

\def\F{{\mathcal F}}
\def\E{{\mathbb E}}
\def\P{{\mathbb P}}
\def\G{{\mathbb G}}
\def\C{{\mathcal C}}
\def\HH{{\mathcal H}}
\def\Hb{{\mathbb H}}

\def\Tr{{\mathrm Tr}}
\def\P{{\mathbb P}}

\newtheorem{thm}{Theorem}[section]

\newtheorem{lemma}[thm]{Lemma}
\newtheorem{prop}[thm]{Proposition}

\newcommand{\red}[1]{\textcolor{red}{#1}}
\newcommand{\Mod}[1]{\ \mathrm{mod}\ #1}
\newenvironment{subproof}[1][\proofname]{%
  \begin{proof}[#1]%
}{%
  \end{proof}%
}
\title{Upper tail bounds for cycles}
\author{Abigail Raz\thanks{Department of Mathematics, Rutgers University, Piscataway NJ. Email: ajr224@math.rutgers.edu}}
\date{}
\begin{document}
\maketitle
\begin{abstract}
This paper examines bounds on upper tails for cycle counts in $G_{n,p}$. For a fixed graph $H$ define $\xi_H= \xi_H^{n,p}$ to be the number of copies of $H$ in $G_{n,p}$. It is a much studied and surprisingly difficult problem to understand the \emph{upper tail} of the distribution of $\xi_H$, for example, to estimate
\begin{equation*}
    \pr(\xi_H > 2 \E\xi_H).
\end{equation*}
The best known result for general $H$ and $p$ is due to Janson, Oleszkiewicz, and Ruci{\'n}ski, who, in 2004, proved
\begin{align}\label{a:JOR}
\exp[-O_{H, \eta}(M_H(n,p) \ln(1/p))]&<\pr(\xi_H > (1+\eta)\E \xi_H)\\&<\exp[-\Omega_{H, \eta}(M_{H}(n,p))].\nonumber
\end{align}
Thus they determined the upper tail up to a factor of $\ln(1/p)$ in the exponent. There has since been substantial work to improve these bounds for particular $H$ and $p$. We close the $\ln(1/p)$ gap for cycles, up to a constant in the exponent. Here the lower bound in (\ref{a:JOR}) is the truth for $l$-cycles when $p> \frac{\ln^{1/(l-2)}n}{n}$. 
\end{abstract}

\section{Introduction}
Let $\G = G(m,p)$ be the usual (Erd\H{o}s-R{\'e}nyi) random graph. A \emph{copy} of $H$ in $\G$ is a subgraph of $\G$ isomorphic to $H$. It is a much-studied question to estimate, for $\eta>0$ and $\xi_H= \xi_H^{m,p}$ the number of copies of $H$ in $G_{m,p}$,
\begin{equation}\label{mainprob}
    \pr(\xi_H > (1+\eta)\E \xi_H).
\end{equation}
To avoid irrelevancies we will always assume $p \ge m^{-1/m_H}$, where (see \cite[pg. 56]{JLR})
\begin{equation*}
    m_H = \max\{e_K/v_K: K \subseteq H, v_K>0\}.
\end{equation*}
(So in the case of cycles we assume $p \ge m^{-1}$.) Then $m^{-1/m_H}$ is a threshold for $``G \supseteq H"$ (see \cite[Theorem 3.4]{JLR}). For smaller $p$ (and bounded $\eta$) the quantity in (\ref{mainprob}) is $\Theta(\min\{m^{v_K}p^{e_K}: K \subseteq H, e_K >0\})$ (see \cite[Theorem 3.9]{JLR} for a start).

Investigation the distribution of $\xi_H$ began in 1960 with Erd{\H o}s and R{\'e}nyi \cite{ER}. In the case of triangles it is easy to see that the upper tail is lower bounded by $\exp[-O(n^2p^2\ln(1/p))]$ (since this is the probability that $G_{n,p}$ contains a complete graph on, say, $2np$ vertices). This is, usually, much bigger than the naive guess, $\exp[-\Omega(n^3p^3)]$, a first indication that the problem is hard. In fact, not much was known about the upper tail until 2000 when 
Vu proved the first exponential tail bound in \cite{Vurg}. 
More information on what was known prior to 2002 can be found in \cite{JR}. A breakthrough occurred in 2004 when, in \cite{KVrg}, Kim and Vu showed, using the ``polynomial concentration method" of \cite{KVpoly}, that when $H$ is a triangle and $p> \frac{\log m}{m}$,
\begin{equation*}
    \pr(\xi_H>(1+\eta)\E\xi_H)<\exp[-\Omega_{\eta}(m^2p^2)].
\end{equation*}

The Kim-Vu bound for triangles was vastly extended by Janson, 
Olesz\-kiewicz, and Ruci{\'n}ski in 2004. To state their result we require the following definition:
\begin{equation*}
    M_H(m,p)=\begin{cases}
    m^2p^{\Delta_H} & \text{ if }p\ge m^{-1/\Delta_H}\\
    \min_{K \subseteq H}\{m^{v_K}p^{e_K}\}^{1/\alpha^*_K} & \text{ if }m^{-1/m_H}\le p <m^{-1/\Delta_H}.
    \end{cases}
\end{equation*}
(As usual $\alpha^*$ is fractional independence number (see e.g.\ \cite{Boll}) and $\Delta_H$ is maximum degree.)
\begin{thm}\cite[Theorem 1.2]{JOR}\label{JOR}
For any $H$ and $\eta$,
\begin{align}
    \exp[-O_{H,\eta}(M_H(m,p)\ln(1/p))]&<\pr(\xi_H>(1+\eta)\E\xi_H) \nonumber \\&<\exp[-\Omega_{H,\eta}(M_H(m,p))].\label{JOReq}
\end{align}
\end{thm}
\noindent (Note, $M_H(m,p)$ is not quite the quantity $M^*_H(m,p)$ used in \cite{JOR}, but as shown in their Theorem 1.5, the two quantities are equivalent up to a constant factor; so the difference is irrelevant here.) 

Thus they determined the probability in (\ref{mainprob}) up to a factor of $O(\ln(1/p))$ in the exponent for constant $\eta>0$. This remains the best result for general $H$ and $p$.
The first progress towards closing the $\ln(1/p)$ gap was made by Chatterjee in \cite{Chat} and DeMarco and Kahn in \cite{DK1} who independently closed it for triangles, showing that, for $p>\log m/m$, the lower bound is the truth (up to the constant in the exponent). DeMarco and Kahn also gave the order of the exponent for smaller $p>1/m$ where the lower bound in (\ref{JOReq}) (namely $\exp[-\Omega(n^2p^2 \ln n)]$) is no longer the answer. Later, in \cite{DK2}, DeMarco and Kahn closed the gap for $l$-cliques, showing that (for $p\ge m^{-2/(l-1)}$, $\eta>0$, and $l>1$)
\begin{equation*}
    \pr(\xi_{K_l}> (1+\eta)\E\xi_{K_l})<\exp[-\Omega_{l,\eta}(\min\{m^2p^{l-1}\log(1/p),m^lp^{l \choose 2}\})].
\end{equation*}
When $H$ is a ``strictly balanced" graph and $p$ is small ($p\le m^{-v/e}\log^{C_h}m$). Warnke, in \cite{War}, used a combinatorial sparsification idea based on the BK inequality \cite{BK,BKR} to close the $\ln(1/p)$ gap, improving on work in \cite{Vurg,Sil}. There was a breakthrough in 2016 when Chatterjee and Dembo introduced a ``nonlinear large deviation" framework \cite{CD}. This has been used to close the gap for general $H$ and large $p$ (i.e. $p>m^{-\alpha_H}$) \cite{CD,LZ}. Recently this technique was used, in \cite{CookD}, by Cook and Dembo to close the gap --- including determining the correct constant in the exponent --- for cycles when $p\gg m^{-1/2}$ (among other results). Additionally, outside of the large deviation framework, Warnke and \v{S}ileikis, in \cite{SW}, recently determined the correct upper tail bound for stars (including in the case where $\eta \ge n^{-\alpha}$ rather than a constant).

Here we settle the question for cycles (i.e.\ the order of magnitude of the exponent), where, with the $l$-cycle denoted $C_l$,
\begin{equation*}
    M_{C_l}(m,p)=m^2p^2.
\end{equation*}
Formally, letting $\xi_l=\xi_l(\G)$ be the number of copies of $C_l$ in $\G$ we prove:

\begin{thm} \label{mainthm}
For any fixed $l$, $\eta>0$, and $p\in [0,1]$,
$$\pr(\xi_l > (1+ \eta) \E \xi_l) < \exp[-\Omega_{\eta,l}( \min \{m^2p^2\ln(1/p),m^lp^l\})].$$
\end{thm}
We are most interested in the range where $m^2p^2\ln(1/p)<m^lp^l$, so essentially when $p> \frac{\ln^{1/(l-2)}m}{m}$.
As in \cite{DK1}, it is convenient to work with an $l$-partite version of the random graph. Let $\Hb$ be the random $l$-partite graph on $ln$ vertices where the vertex set is the disjoint union of $l$ $n$-sets, say $V=V(\Hb)=V_1\cup \cdots \cup V_l$, and $\pr(xy \in E(\Hb))=p$ whenever $x \in V_i$ and $y \in V_{i+1}$ for some $i$ (all subscripts$\Mod{l}$), these choices made independently. There are no edges between other pairs $(V_i,V_j)$ or within a $V_i$. We always take $v_i$ to be a vertex of $V_i$. A \emph{copy} of $C_l$ in $\Hb$ is any subgraph, with vertices $v_1, v_2, \ldots, v_l$ isomorphic to $C_l$. Note these are not all of the subgraphs of $H$ isomorphic to $C_l$ since we demand each vertex of the cycle is in a different $V_i$. We denote the number of copies of $C_l$ in $\Hb$ by $\xi_l'$. A \emph{copy} of the $l-1$ path (denoted $P_{l-1}$) is any path $v_1, v_2, \ldots, v_l$ isomorphic to $P_{l-1}$ (i.e. $v_i \sim v_{i+1}$ for $1\le i <l$). We use $(v_1, \ldots, v_l)$ to denote both copies of $C_l$ and copies of $P_{l-1}$, since it will always be clear which interpretation is intended. We show the following bound.  
\begin{thm}\label{realthm}
For any fixed $l$, $\delta>0$, and $p\in [0,1]$, 
\begin{equation}\label{realeq}
    \pr(\xi'_l > (1+ \delta) n^lp^l) < \exp[-\Omega_{\delta,l}( \min \{n^2p^2\ln(1/p),n^lp^l\})].
\end{equation}
\end{thm}
That Theorem \ref{realthm} implies Theorem \ref{mainthm} is likely well known and an easy generalization from the $l=3$ case which can be found in \cite{DK1}. However, for completeness we will still give the general argument.
\begin{prop}\label{imp}
Theorem \ref{realthm} implies Theorem \ref{mainthm}.
\end{prop}
This is proved in Section \ref{reduction}. The rest of the paper is organized as follows. Section \ref{main} gives notation and states the two main assertions that give Theorem \ref{realthm}. These are proved in Sections \ref{pf1}-\ref{pf2}, with Section \ref{prelim} devoted to preliminaries.

\section{Reduction}\label{reduction}

For completeness we give the proof of Proposition \ref{imp}, following \cite{DK1}.

\begin{proof}[Proof of Proposition \ref{imp}]
We first claim that it is enough to prove Proposition \ref{imp} for $m=ln$. Assuming we know Proposition \ref{imp} for $m=ln$ we show it still holds when $m = -k \mod l$. Given $\eta$ and $l$, we may assume $m$ is large (formally $m>m_{\eta,l}$). So, for example,
$$(1+\eta){m \choose l}> (1+\eta/2){m+k \choose l}.$$
Therefore,
\begin{align*}
    \pr\left(\xi_l > (1+\eta){m \choose l}p^l\right)&\le \pr\left(\xi_l >(1+\eta/2){m+k \choose l}p^l\right)\\
    &<\exp[-\Omega_{\eta/2, l}( \min\{(m+k)^2p^2 \ln(1/p), (m+k)^lp^l\})]\\
    &=\exp[-\Omega_{\eta,l}(\min\{m^2p^2\ln(1/p), m^lp^l\})].
\end{align*}
Note the second inequality holds since $m+k$ is a multiple of $l$. \\Now to prove Proposition \ref{imp} when $m=ln$ let $\eta$ be as in Theorem \ref{mainthm}, and set $\delta= \frac{\eta}{2+\eta}$. We can choose $\Hb$ by first choosing $\G$ on $V= [ln]$ and then selecting a uniform equipartition $V_1 \cup \cdots \cup V_l$, and setting $$E(\Hb)= \{xy \in E(\G): x,y \text{ belong to consecutive }V_i's\}.$$ Note that, for any possible value $G$ of $\G$
\begin{equation}\label{lpart1}
    \E[\xi' | \G=G]= \rho \xi(G),
\end{equation}
where $\rho = n^l/{ln \choose l}$. On the other hand, letting $$\alpha(G)= \pr(\xi' < (1-\delta)\rho \xi(G)|\G=G),$$ we have
\begin{equation}\label{lpart2}
    \E[\xi'|\G=G]\le \alpha(G)(1-\delta)\rho \xi(G)+(1-\alpha(G))\xi(G).
\end{equation}
Combining (\ref{lpart1}) and (\ref{lpart2}) gives $\alpha(G)\le 1- \frac{\delta \rho}{1-\rho +\delta \rho}\coloneqq 1-\beta$. We also have, by Theorem \ref{realthm}, 
\begin{equation*}
    \exp[-\Omega_{\delta,l} (\min\{n^2p^2\ln(1/p), n^lp^l\})] > \pr(\xi_l' >(1+\delta)n^lp^l).
\end{equation*}
Additionally, we know
\begin{align*}
    &\pr(\xi_l' >(1+\delta)n^lp^l)\\
    &\ge \pr\left(\xi'_l> (1+\delta)n^lp^l | \xi_l> \frac{1+\delta}{1-\delta}{ln \choose l}p^l\right)\pr\left(\xi_l> \frac{1+\delta}{1-\delta}{ln \choose l}p^l\right)\\
    & \ge \beta \pr\left(\xi_l > \frac{1+\delta}{1-\delta}{ln \choose l}p^l\right).
\end{align*}
Here the final inequality holds since $(1-\delta)\rho\frac{1+\delta}{1-\delta}{ln \choose l}p^l=(1+\delta)n^lp^l$ and, as we showed, $\alpha(G)$ is always at most $(1-\beta)$. Since $\frac{1+\delta}{1-\delta}= 1+\eta$, Theorem \ref{mainthm} follows.
\end{proof}


\section{Main Lemmas}\label{main}
Recall that we always take $v_i$ to be a vertex in $V_i$; indices are always written $\Mod l$; and \emph{copy} of $C_l$, \emph{copy} of $P_{l-1}$ were defined just before the statement of Theorem \ref{realthm}.  
We use $\C$ to denote the set of copies of $C_l$ in $\Hb$. Additionally, we abusively use just \emph{cycle} for ``copy of $C_l$'' and \emph{full path} for ``copy of $P_{l-1}$''. As usual $N_{Y}(x)= \{y \in Y : xy \in E(\Hb)\}$, $d_Y(x)= |N_Y(x)|$, $d(x,y) = |N_V(x)\cap N_V(y)|$, and $\Delta$ is the maximum degree in $\Hb$ (we also use $N(x)= N_V(x)$ and $d(v)= d_V(x)$). Let 
\begin{equation*}
  \hat{d}(v_i)= \max\{d_{V_{i-1}}(v_i),d_{V_{i+1}}(v_i)\}.  
\end{equation*}
We will abusively refer to $\hat{d}(v)$ as the \emph{degree} of $v$.
For disjoint $X,Y \subseteq V$ we use $\nabla(X)$ (resp. $\nabla(X,Y)$) for the set of edges with one end in $X$ (resp. one end in each of $X,Y$). 

Much of the set-up that follows is borrowed from or inspired by \cite{DK1}. Set $t= \ln (1/p)$ and $s= \min\{t,n^{l-2}p^{l-2}\}$ (so the exponent in (\ref{realeq}) is $-\Omega_{\delta,l}(n^2p^2s)$). For simplicity set $\gamma= \frac{1}{5l^2}$ and
\begin{equation}\label{epsilon}\epsilon = \frac{\delta}{(27l)^{l+1}}. \end{equation} 

Note that for a fixed $\nu$ and $p>\nu$, Theorem \ref{mainthm} is covered by Theorem \ref{JOR}. For us it is convenient to pick $\nu= e^{-4/\gamma}=e^{-20l^2}$. Of course, the partite version (Theorem \ref{realthm}) was not considered in \cite{JOR}, but it is not too hard to get this from Theorem \ref{JOR}: 
\begin{prop}\label{JORimp}
For $p>e^{-20l^2}$ Theorem \ref{realthm} follows from Theorem \ref{JOR}.
\end{prop}
\noindent This will be proved at the end of the section. 

In view of Proposition \ref{JORimp}, we may assume for the proof of Theorem \ref{realthm} that 
\begin{equation}\label{pupper}
  p\le e^{-4/\gamma}=e^{-20l^2}.  
\end{equation}

We may also assume: $\delta$ --- so also $\epsilon$ --- is (fixed but) small (since (\ref{realeq}) becomes weaker as $\delta$ grows); given $\delta$ and $l$, $n$ is large (formally, $n> n_{\delta,l}$); and, say,
\begin{equation}\label{plower}
    p>\epsilon^{-4}n^{-1}
\end{equation}
(since for smaller $p$, Theorem \ref{realthm} is trivial for an appropriate $\Omega_{\delta,l}$).
We say that an event occurs \emph{with large probability} (w.l.p.) if its probability is at least $1-\exp[-T\epsilon^4n^2p^2t]$ for some fixed $T>0$ and small enough $\epsilon$. We write ``$\alpha <^* \beta$" for ``w.l.p. $\alpha < \beta$". Note that, assuming (\ref{plower}), an intersection of $O(n)$ events that hold w.l.p. also holds w.l.p.

Let $V_i' = \{v \in V_i : \hat{d}(v)< np^{1-\gamma}\}$ and let $f(v_1,v_l)$ be the number of full paths with endpoints $v_1$ and $v_l$ in which each vertex is in the appropriate $V_i'$. 

\noindent The next two assertions imply Theorem \ref{realthm}: 
\begin{equation} \label{case1}
    \text{w.l.p. } |\{(v_1, \ldots, v_l) \in \C: \exists i (v_i \notin V'_i )\}|< (\delta/2) n^lp^l;
\end{equation}
\begin{equation} \label{case2}
    \pr(|\{(v_1, \ldots, v_l) \in \C: \forall i (v_i \in V'_i )\}|>(1+\delta/2)n^lp^l)< \exp[- \Omega_{\delta,l}(n^2p^2s)].
\end{equation}

\noindent We prove (\ref{case1}) in Section \ref{pf1} and (\ref{case2}) in Section \ref{pf2}. In Section \ref{pf3} we prove that 
\begin{equation}\label{pathtotal}
\sum_{v_1,v_l} f(v_1,v_l)<^* (1+\delta/8)n^lp^{l-1},
\end{equation}
which will be used in the proof of (\ref{case2}).

We now give the proof of Proposition \ref{JORimp}. To do so we require the following tail bound due to Janson (\cite{Jan}; see also \cite[Theorem 2.14]{JLR}).
\begin{lemma}\label{janson}
Let $\Gamma$ be a set of size $N$ and $\Gamma_p$ the random subset of $\Gamma$ in which each element is included with probability $p$ (independent of the other choices). Assume $\mathcal{S}$ is a family of non-empty subsets of $\Gamma$, and for each $A \in \mathcal{S}$ let $I_A= 1[A \subseteq \Gamma_p]$. Additionally, let $X = \sum_{A \in \mathcal{S}}I_A$. Define 
\begin{equation*}
\bar{\Delta}= \sum \sum_{A\cap B \neq \emptyset}\E(I_AI_B).
\end{equation*}
Then for $0\le t \le \E X$,
\begin{equation*}
    \pr(X\le \mu-t)\le \exp\left[\frac{-t^2}{2\bar{\Delta}}\right].
\end{equation*}
\end{lemma}

\begin{proof}[Proof of Proposition \ref{JORimp}]
Let $\Hb$ be as in Theorem \ref{realthm} and regard $\Hb$ as a subgraph of $\G= G_{ln,p}$. Set $\xi=\xi_l(\G)$, $\xi'=\xi'_l(\Hb)$, and $\xi''=\xi-\xi'$; thus $\xi''$ is the number of cycles in $\G$ that are not of the form $(v_1, \ldots, v_l)$. Then $\E[\xi'']= \left(\frac{(ln)!}{(ln-l)!2l}-n^l\right)p^l$. We first use Lemma \ref{janson} to show 
\begin{equation*}
     \pr(\xi''<(1-\epsilon)\E\xi'')\le \exp[-\Omega_{l,\epsilon}(n^2)].
\end{equation*}
To apply Lemma \ref{janson} we take $\mathcal{S}$ to be the set cycles in $G$ not of the form $(v_1, \ldots, v_l)$ (so each $A \in \mathcal{S}$ is the edge set of a particular cycle). Note that when $|A \cap B|=k$ we have $\E[I_A I_B]=p^{2l-k}$. Furthermore, the number of pairs of cycles sharing exactly $k\ge 1$ edges is at most $c^k_l n^{2l-(k+1)}$ (for some constants $c_l^k$). Thus we have
\begin{equation*}
\bar{\Delta}\le \sum_k c^k_l n^{2l-(k+1)}p^{2l-k}=c_ln^{2l-2},
\end{equation*}
since $p =\Omega(1)$. Lemma \ref{janson}, with $t= \epsilon \E \xi''$, gives 
\begin{equation}\label{xi''}
    \pr(\xi''<(1-\epsilon)\E\xi'')\le \exp[-\Omega_{l,\epsilon}(n^2)].
\end{equation}
Furthermore, we claim that for any $\delta'>0$
\begin{equation}\label{highpeq}
    \pr(\xi'>(1+\delta')\E\xi')\le \pr(\xi''<(1-\delta'')\E \xi'')+\pr(\xi>(1+\delta)\E \xi),
\end{equation}
provided $\delta$ and $\delta''$ are such that $\delta \E \xi + \delta'' \E \xi''< \delta' \E \xi'$. 
This is because occurrence of the event on the l.h.s.\ implies occurrence of one of the events on the r.h.s.\ ; namely, if 
\begin{equation*}
\xi''\ge(1-\delta'')\E \xi'' \hspace{.5 in} \text{and} \hspace{.5 in} \xi \le (1+\delta)\E \xi,
\end{equation*}
then
\begin{align*}
    \xi' =\xi-\xi''
    &\le (1+\delta)\E \xi-(1-\delta'')\E \xi''\\
    &= \E \xi'+\delta \E \xi + \delta'' \E \xi''\\
    &<(1+\delta')\E \xi'.
\end{align*}

Therefore, for any $\eta>0$ we can select $\delta$ and $\delta''$ such that 
\begin{align*}
    \pr(\xi'>(1+\eta)\E\xi')&\le \pr(\xi''<(1-\delta'')\E \xi'')+\pr(\xi>(1+\delta)\E \xi)\\
    &<\exp[-\Omega_{\delta'',l}(n^2)]+\pr(\xi>(1+\delta)\E \xi)\\
    &<\exp[-\Omega_{\delta'',l}(n^2)]+\exp[-\Omega_{\delta,l}(n^2)],
\end{align*}
where the second inequality holds by (\ref{xi''}) and the third by Theorem \ref{JOR}.
\end{proof}

\section{Preliminaries}\label{prelim}
To prove (\ref{case1}) and (\ref{case2}) we need the following preliminaries, where $B(m, \alpha)$ is used for a random variable with the binomial distribution $\text{Bin}(m,\alpha)$. The first two of these are standard large deviation bounds; see e.g.\ \cite[Theorem A.1.12]{AS}, \cite[Theorem 2.1(a)]{JLR} and \cite[Lemma 8.2]{BC}. The others are applications of Lemma \ref{basic} that we will use repeatedly. 

\begin{lemma}\label{basic}

For any $\beta\in (0,1)$, $K\ge1+\beta$, $m$, and $\alpha$ we have,
\begin{equation}\label{basic1}
    \pr(B(m, \alpha)\ge Km \alpha)<\begin{cases}
    \exp[-\beta^2m\alpha/4] &\mbox{if $K\le 4$,}\\
    (e/K)^{Km \alpha} & \mbox{if $K>4$.}
    \end{cases}
\end{equation}

\end{lemma}
\noindent When $m=n$ and $\alpha=p$ (which is what we have when our binomial random variable is $d_{V_{i-1}}(v_{i})$ or $d_{V_{i+1}}(v_i)$) and $K\ge1+\epsilon$ we use $q_K$ for the right hand side of (\ref{basic1}); that is,
\begin{equation}\label{qk}
q_K\coloneqq \begin{cases}
    \exp[-\epsilon^2np/4] &\mbox{if $K\le 4$,}\\
    (e/K)^{Knp} & \mbox{if $K>4$.}
    \end{cases}
\end{equation} 
First note that for any $K$ ($\ge 1+\epsilon$) we have,
\begin{equation}\label{weakqk}
    q_K\le \exp[-\epsilon^2Knp/16].
\end{equation}
Of course this is unnecessarily weak when $K$ is not close to 1 (as was the first bound in (\ref{basic})), but is often enough for our purposes and will be used repeatedly below.
It will also be useful to have the following upper bound on $q_K$ when $K\ge p^{-\gamma/2}$ (recall $\gamma$ was defined before (\ref{epsilon})):
\begin{equation}\label{qkbound}
    q_K \le \exp[-\gamma Knpt/4]<n^{-2}.
\end{equation}
To show the first inequality holds note that $K \ge p^{-\gamma/2}$ and $p\le e^{-4/\gamma}$ (see (\ref{pupper})) imply $K\ge e^{2}$ and \begin{equation*}
   (e/K)^{Knp}\le \exp\left[Knp\left(1-\frac{\gamma}{2}t\right)\right]. 
\end{equation*}
Again $p\le e^{-4/\gamma}$ implies $t\ge 4/\gamma$ giving the first inequality in (\ref{qkbound}):
\begin{equation*}
q_K=(e/K)^{Knp}\le\exp[-\gamma Knpt/4].
\end{equation*} 
The second inequality in (\ref{qkbound}) follows easily from the combination of $t \ge 4/\gamma$ and the fact that $p$ is not extremely small (see (\ref{plower})).

\begin{lemma}\label{sum}
Suppose $w_1, \ldots, w_m \in [0,z]$. Let $\zeta_1, \ldots, \zeta_m$ be independent Ber\-noullis, $\zeta= \sum \zeta_i w_i$, and $\E \zeta =\mu$. Then for any $\nu >0$ and $\lambda> \nu \mu$,
\begin{equation*}
    \P(\zeta> \mu+\lambda)< \exp[-\Omega_\nu(\lambda/z)].
\end{equation*}
\end{lemma}

The last two lemmas are the basis for much of what follows. Lemma \ref{degsum} in particular may be regarded as perhaps the main idea for sections \ref{pf1} and \ref{pf3}; it allows us to bound sums of atypically large degrees, which we then use to bound the number of cycles that include vertices of ``large'' degree (in Section \ref{pf1}) and the number of full paths without vertices of ``large'' degree (in Section \ref{pf3}).

\begin{lemma}\label{degsize}
For $K \ge 1+\epsilon$ and any $i$, 
\begin{equation}\label{degsize1}
|\{v_i \in V_i : \hat{d}(v_i)\ge K np\}|<^*r_K \coloneqq \begin{cases}
6\epsilon K^{-l}n & \text{if } q_K> n^{-2},\\ 
\frac{\epsilon^2 npt}{K \ln K} &\text{otherwise}. 
\end{cases}
\end{equation}
\end{lemma}
\noindent The first, \emph{ad hoc} value is for use in Section \ref{pf3} while the second will be used throughout. Convenient bounds for the second expression in (\ref{degsize1}) are
\begin{equation}\label{degsize2}
    \frac{\epsilon^2 npt}{K \ln K}< \begin{cases}
    2\epsilon npt/K & \text{if }K> 1+\epsilon,\\
    \epsilon np/K & \text{if }K> p^{-\epsilon}.
    \end{cases}
\end{equation}

\begin{proof}[Proof of Lemma \ref{degsize}]
Let $q=q_K$ and $r=\min\{r_K,1\}$. We let $r=\min\{r_K,1\}$ because later it will be helpful to have $n/r \le n$. We can enforce this lower bound on $r$ because if $r_K<1$ then 
\begin{equation*}
    \pr(|\{v_i \in V_i : \hat{d}(v_i)\ge K np\}|\ge r)= \pr(|\{v_i \in V_i : \hat{d}(v_i)\ge K np\}|\ge 1).
\end{equation*} Without loss of generality, let $i=1$. We show 
\begin{equation}\label{degsizepf1}
    |\{v_1 \in V_1 : d_{V_2}(v_1)\ge K np\}|<^*r/2.
\end{equation}
Write $N$ for the left hand side of (\ref{degsizepf1}).
We first assume $q \le n^{-2}$. Since the $d_{V_2}(v_1)$'s ($v_1\in V_1$) are independent copies of $B(n,p)$, two applications of Lemma \ref{basic} give 
\begin{align*}
    \pr(N \ge r)&< \pr(B(n,q)\ge \lceil r/2 \rceil)\\
    &<(2enq/r)^{r/2}\\
    &\le(2e\sqrt{q})^{r/2} \\
    &< \exp[-\Omega(\epsilon^4n^2p^2 t)].
\end{align*}
The third inequality holds since $q\le n^{-2}$, so $n/r \le n\le q^{-1/2}$. \\
Now assume $q>n^{-2}$. Recall from (\ref{weakqk}) that we always have 
\begin{equation*}
    q\le \exp[-\epsilon^2Knp/16].
\end{equation*}
So, $$n^{-2}<q\le \exp[-\epsilon^2Knp/16]$$ implies
\begin{equation}\label{degsizepf}
    Knp<32\epsilon^{-2}\log n,
\end{equation}
On the other hand (\ref{plower}) gives 
\begin{equation*}
q < \exp[-\epsilon^2Knp/16]< \exp[-\epsilon^{-2}K/16]<\epsilon K^{-l}.
\end{equation*} The last inequality uses the fact that $\exp[\epsilon^{-2}K/16]\epsilon K^{-l}$ is minimized at $K=16l\epsilon^2$ and $\epsilon<\left(\frac{e}{16l}\right)^{l/(2l-1)}$ (as we may assume). Hence
$$\pr(N\ge r/2) <\pr(B(n,q)\ge r/2)<\exp[-\Omega(\epsilon n K^{-l})]<\exp[-\Omega(n^2p^2t)],$$
where the second inequality uses $r/2> 3nq$ (and Lemma \ref{basic}) and the (very crude) third inequality uses $K^{l-2}<n/\log^3n$ which follows from (\ref{degsizepf}) and (\ref{plower}).
\end{proof}

\begin{lemma}\label{degsum}
For $p> \frac{64\epsilon^{-2}\ln n}{n}$ and any $i$, 
\begin{equation} \label{degsumeq}
    \sum \left\{\hat{d}(v_i):\hat{d}(v_i)>(1+\epsilon) np\right\} <^* 
    \epsilon^2 n^2p^2t, 
\end{equation}
and
\begin{equation}\label{degsumeq2}
    \sum \left\{\hat{d}(v_i):\hat{d}(v_i)>np^{1-\gamma/2}\right\}<^* \epsilon n^2p^2.
\end{equation}
\end{lemma}
There is nothing special about $\gamma/2$ here; it is simply a value that will work for our purposes. The reason for the particular --- and not very important --- lower bound on $p$ will appear following (\ref{degsumpf2}).
\begin{proof}
First we show (\ref{degsumeq}). To slightly lighten the notation we fix $i$ and set
\begin{equation*}
    W=\{v_i : \hat{d}(v_i)> (1+\epsilon)np\}.
\end{equation*}
We partition $W=\bigcup_{j=0}^J W^j$ (where $J \coloneqq \log_2((p(1+\epsilon))^{-1})-1<2t$), with
$$W^j= \{v_i: 2^j(1+\epsilon)np< \hat{d}(v_i)\le 2^{j+1}(1+\epsilon)np\}.$$ It suffices to show 
\begin{equation}\label{sumets}
\sum_{j=0}^J |W^j|2^{j+1}(1+\epsilon)np <^* \epsilon^2n^2p^2t.
\end{equation}
Lemma \ref{basic} (using just (\ref{weakqk})) gives 
\begin{align*}
    \pr(v_i \in W^j)& \le \pr(\hat{d}(v_i)> 2^j (1+\epsilon)np)\\
    &\le 2\exp[-\epsilon^2 2^{j-4}np]< \exp[-\epsilon^2 2^{j-5}np].
\end{align*}
Thus, for any $(a_0, \ldots, a_J)$,
\begin{align}
    \pr\left(|W^0|=a_0,\ldots, |W^J|=a_j\right) & < \exp\left[\sum_{j=0}^J - a_j \epsilon^2 2^{j-5}np\right]\prod_{j=0}^J{n \choose a_j} \nonumber \\
& <\exp\left[\sum_j a_j(\ln n- \epsilon^2 2^{j-5}np)\right] \nonumber \\
& \le \exp\left[\sum_j -a_j \epsilon^2 2^{j-6}np\right].\label{degsumpf2}
\end{align}
For (\ref{degsumpf2}) we note that $p > \frac{64\epsilon^{-2}\ln n}{n}$, so $\epsilon^2 2^{j-5}np\ge 2\ln n$. 

On the other hand, for (\ref{sumets}) it is enough to show 
\begin{equation}\label{asat1}
\sum_{(a_0, \ldots, a_J)}\pr\left(|W^0|=a_0,\ldots, |W^J|=a_j\right)<\exp[-T\epsilon^4n^2p^2t]
\end{equation}
for some constant $T>0$ (not depending on $\epsilon$), where we sum over $(a_0, \ldots, a_J)$ satisfying
\begin{equation}\label{asat2}
    \sum_j a_j2^{j+1}(1+\epsilon)np>\epsilon^2n^2p^2t.
\end{equation}
Here we can just bound the number of terms in (\ref{asat1}) by the trivial 
\begin{equation*}
n^J< \exp[2t\log n],
\end{equation*} while (in view of (\ref{asat2})) (\ref{degsumpf2}) bounds the individual summands in (\ref{asat1}) by $$\exp[-\Omega(\epsilon^4n^2p^2t)].$$ Moreover, the lemma's lower bound on $p$ (or the weaker $p \gg \frac{\log^{1/2}n}{n}$) implies $n^2p^2t \gg t \log n$. So the left hand side of (\ref{asat1}) is at most
\begin{equation*}
    \exp[2t \log n -\Omega(\epsilon^4n^2p^2t)]= \exp[-\Omega(\epsilon^4n^2p^2t)],
\end{equation*}
as desired.

To show (\ref{degsumeq2}) we now let $W = \{v_i : \hat{d}(v_i)> np^{1-\gamma/2}\}$. As before, we partition $W=\bigcup_{j=0}^J W^j$ (where $J \coloneqq \log_2(p^{-1+\gamma/2})-1<2t$) with
$$W^j= \{v_i: 2^jnp^{1-\gamma/2}< d(v_i)\le 2^{j+1}np^{1-\gamma/2}\}.$$ It suffices to show 
\begin{equation}\label{sumets2}
\sum_{j=0}^J |W^j|2^{j+1}np^{1-\gamma/2} <^* \epsilon n^2p^2.
\end{equation}
Lemma \ref{basic} and (\ref{qkbound}) give 
\begin{align*}
\pr(v_i \in W^j) &\le \pr(\hat{d}(v_i)> 2^j np^{1-\gamma/2})\\
&\le 2\exp[- \gamma 2^{j-2}np^{1-\gamma/2}t]< \exp[-\gamma 2^{j-3}np^{1-\gamma/2}t].
\end{align*}
Thus, for any $(a_0, \ldots, a_J)$,
\begin{align}
    \pr\left(|W^0|=a_0,\ldots, |W^J|=a_J\right) & <\exp\left[\sum_{j=0}^J - a_j \gamma 2^{j-3} np^{1-\gamma/2}t\right] \prod_{j=0}^J {n \choose j}\nonumber \\
& <\exp\left[\sum_j a_j(\ln n- \gamma 2^{j-3}np^{1-\gamma/2}t)\right] \nonumber \\
& <\exp\left[\sum_j -a_j \gamma 2^{j-4}np^{1-\gamma/2}t\right].\label{degsumpf1}
\end{align}
((\ref{degsumpf1}) follows from $\gamma 2^{j-3}np^{1-\gamma/2}t \gg \ln n$, in this case a very weak consequence of our assumed lower bound on $p$.)

For (\ref{sumets2}) it is enough to show 
\begin{equation}\label{asat3}
\sum_{(a_0, \ldots, a_J)}\pr\left(|W^0|=a_0,\ldots, |W^J|=a_j\right)<\exp[-T\epsilon^4n^2p^2t]
\end{equation}
for some constant $T>0$ (not depending on $\epsilon$) where we sum over $(a_0, \ldots, a_J)$ satisfying
\begin{equation}\label{asat4}
    \sum_j a_j2^{j+1}np^{1-\gamma/2}>\epsilon n^2p^2.
\end{equation}
Again we can just bound the number of terms in (\ref{asat3}) by the trivial 
\begin{equation*}
n^J< \exp[2t\log n],
\end{equation*} while (in view of (\ref{asat4})) (\ref{degsumpf1}) bounds the individual summands by $$\exp[-\Omega(\epsilon n^2p^2t)].$$
Again since the lemma's lower bound on $p$ (or the weaker $p \gg \frac{\log^{1/2}n}{n}$) implies $n^2p^2t \gg t \log n$, the left hand side of (\ref{asat3}) is at most
\begin{equation*}
    \exp[2t \log n -\Omega(\epsilon n^2p^2t)]= \exp[-\Omega(\epsilon n^2p^2t)]],
\end{equation*}
as desired.
\end{proof}


We will also make use of the fact that for any $\beta>0$, $k$, and $p$,
\begin{equation}\label{ptbound}
    p^\beta \ln^k(1/p) \le \left(\frac{k}{e\beta}\right)^k.
\end{equation}
To see this let $f(p)= p^\beta \ln^k(1/p)$, and notice that 
\begin{align*}f'(p)&=-kp^{\beta-1} \ln^{k-1}(1/p)+\beta p^{\beta-1}\ln^k(1/p)\\ &= p^{\beta-1}\ln^{k-1}(1/p)(-k+\beta\ln(1/p)). 
\end{align*} 
Thus $f(p)$ is maximized at $p=e^{-k/\beta}$, where it equals the r.h.s.\ of (\ref{ptbound}).

\section{Proof of (\ref{case1})}\label{pf1}
\begin{proof}[\unskip\nopunct]
We first rule out very small $p$, showing that when 
\begin{equation*}
    p<n^{\frac{-1}{\gamma+1}},
\end{equation*}
\begin{equation}\label{Delta7}
   \text{w.l.p. } \Delta<np^{1-\gamma}, 
\end{equation}
so that (\ref{case1}) is vacuously true. 
For (\ref{Delta7}), with $K= (1/2)p^{-\gamma}$ (and $x$ any vertex), Lemma \ref{basic} (and the union bound) give
\begin{align}
    \pr(\Delta \ge np^{1-\gamma})&\le ln\cdot \pr(d(x)\ge 2Knp) \nonumber \\
    &<ln \cdot \exp[-2Knp\ln(K/e)] \nonumber \\
    &=ln \cdot \exp[-np^{1-\gamma}(\gamma t-\ln(2e))] \label{deltaeq1}.
\end{align}
But for $p<n^{\frac{-1}{\gamma+1}}$ (which is the same as $np^{1-\gamma}>n^2p^2$), the r.h.s.\ of (\ref{deltaeq1}) is $\exp[-\Omega_{\delta,l}(n^2p^2t)]$ (note that (\ref{pupper}) implies $\gamma t\ge 4$ and the initial $ln$ disappears because (\ref{plower}) makes $\gamma n^2p^2t$ a large multiple of $\log n$).
Therefore for the remainder of the proof of (\ref{case1}) we may assume that
\begin{equation}\label{highp7}
p\ge n^{\frac{-1}{\gamma +1}}.
\end{equation}

We say $v$ has \emph{large degree} if $\hat{d}(v)> np^{1-\gamma/2}$ and \emph{intermediate degree} if $np^{1-\gamma/2}\ge\hat{d}(v)>2np$. 
We classify the cycles appearing in (\ref{case1}) according to the positions of their large and intermediate vertices. For disjoint $M,N\subset [l]$, say $v_i$ is of \emph{type} $(M,N)$ if 
\[
\hat{d}(v_i)\left\{\begin{array}{ll}
> np^{1-\gamma/2} &\mbox{if $i \in M$,}\\ 
\in (2np,np^{1-\gamma/2}]&\mbox{if $i \in N$,}\\
\leq 2np&\mbox{otherwise,}
\end{array}\right.
\]
and say a set of vertices is of type $(M,N)$ if each of its members is.
We consider various possibilities for $(M,N)$, always requiring that all vertices under discussion are of the given type. 
To begin note that since we are in (\ref{case1}) we have $M \neq \emptyset$.

A little preview may be helpful. In each case we are trying to show that the size of the set of cycles $(v_1,\ldots, v_l)$ in question is small relative to $m^lp^l$, so would like the number of possibilities for $v_i$ to be, in geometric average, somewhat less than $mp$. For example, for $i \in M$ we do much \emph{better} than this using Lemma \ref{degsize}, which, recall, bounds the number of $v_i$'s of such large degree by $mp^{1+\gamma/2}$ (or $\epsilon mp^{1+\gamma/2}$ but here the $\epsilon$ is minor). On the other hand, for $i \notin M \cup N$ we have only the naive bound $m$, which is clearly unaffordable. To control the number of such $v_i$ we rely on first selecting some $v_{i-1}$ (or $v_{i+1}$) and then bounding the number of choices for $v_i$ by $\hat{d}(v_{i-1})$ (or $\hat{d}(v_{i+1})$). If $i-1,i \notin M \cup N$ then given $v_{i-1}$ we simply use $\hat{d}(v_{i-1})\le 2mp$ as a bound on the number of choices for $v_i$. However if, for example, $i-1\in M \cup N$ and $i \notin M \cup N$ we require Lemma \ref{degsum} to bound the choices for $(v_{i-1}, v_i)$ (with $v_{i-1}\sim v_i$).

\medskip

We now consider cycles of type $(M, \emptyset)$.
Here the absence of intermediate vertices will allow us to relax our assumption that there is at least one vertex of degree at least $np^{1-\gamma}$; we will only need to assume that there is at least one vertex of degree at least $np^{1-\gamma/2}$. Let 
\[
M^* =\{i \in M: i+1 \notin M\},
\]
with subscripts interpreted$\mod l$.
Note that $M \neq \emptyset$ implies $M^*= \emptyset$ only when $M = [l]$. Here and in the future we will tend to somewhat abusively omit ``w.l.p.''\ in situations where this is clearly what is meant.
We will bound:
\begin{enumerate}[label=(\roman*)]
    \item for $i\in M\sm M^*$, the number of possibilities for $v_i$; \label{step1}
    \item for $i\in M^*$, the number of possibilities for $(v_i,v_{i+1})$; \label{step2}
    \item given the choices in \ref{step2}, the number of possibilities for vertices of the cycle 
not chosen in \ref{step1} and \ref{step2}. \label{step3}
\end{enumerate}
Note that the number of vertices chosen in \ref{step3} is $l-|M|-|M^*|$.
The reason for treating $i\in M^*$ in \ref{step2} rather than \ref{step1} is (roughly) that it is through these vertices that we control the number of choices for the vertices that follow them (the $v_{i+1}$'s of \ref{step2}). 
For \ref{step1} we just recall that Lemma \ref{degsize} bounds the number of choices for $v_i$ (of large degree) by $\epsilon np^{1+\gamma/2}$; so the total number of possibilities in \ref{step1} is at most
\begin{equation*}
(\epsilon np^{1+\gamma/2})^{|M|-|M^*|}.
\end{equation*} 

For $i$ as in \ref{step2}, the number of possibilities for $(v_i,v_{i+1})$ is at most
\[
\sum\{ \hat{d}(v_i):\hat{d}(v_i)> np^{1-\gamma/2}\}<^* \epsilon n^2p^2,
\]
with the inequality given by Lemma \ref{degsum}. Thus the total number of possibilities in \ref{step2} is at most
\begin{equation*}
    \left(\epsilon n^2p^2\right)^{|M^*|}.
\end{equation*}

Finally, we may choose the $v_i$'s in \ref{step3} in an order for which each $v_{i-1}$ is chosen before $v_i$ (either because $v_{i-1}$ is chosen in \ref{step2}, or because $i-1 $ precedes $i$ in our order; e.g.\ we can use any cyclic order that begins with an $i$ for which $i-1\in M^*$ --- if $M^*=\emptyset$ then $M = [l]$, so all vertices were chosen in \ref{step1}).
But since $N=\emptyset$, the number of choices for $v_i$ given $v_{i-1}$ is at most $2np$. 

Combining the above bounds we find that, for a given $M$, the number of cycles of type $(M,\emptyset)$ is at most
$$\left(\epsilon n^2p^2\right)^{|M^*|} (\epsilon np^{1+\gamma/2})^{|M|-|M^*|}(2np)^{l-|M|-|M^*|} < \epsilon 2^ln^lp^l<\frac{\delta}{ 2^{l+2}}n^lp^l,$$
(using (\ref{epsilon}) for the last inequality). So, since there are fewer than $2^l$ possibilities for $M$,
\begin{equation}\label{Nempty}
\text{the number of cycles of any type }(M,\emptyset)\text{ is at most }\frac{\delta}{4}n^lp^2.
\end{equation}

Next we consider cycles of type $(M,N)$ with $N \neq \emptyset$. We may assume (at the cost of a negligible factor of $l$ in our eventual bound) that $1 \in N$, and that $k$ is an index for which $\hat{d}(v_k)>np^{1-\gamma}$ (which exists since we are in (\ref{case1}); again, we will pay a factor of $l-1$ for the choice of $k$.) We further define 
\begin{align*}
    N_1&= (N\cup M) \cap \{2, \ldots, k-1\},\\
    N_2&=(N\cup M) \cap \{k+1, \ldots, l\},\\
    N_1^* &=\{i \in N_1\setminus \{k-1\}: i+1 \notin N_1\}, \text{ and }\\
    N_2^* &= \{i \in N_2 \setminus \{k+1\}: i-1 \notin N_2\}.
\end{align*}
We split into cases based on whether $2 \in N_1\cup\{k\}$ and/or $l \in N_2\cup \{k\}$. 
First assume $2 \notin N_1\cup \{k\}$ and $l \notin N_2\cup \{k\}$.
We will bound:
\begin{enumerate}[label=(\roman*)]
    \item the number of possibilities for $v_k$;\label{2step1}
    \item the number of possibilities for $(v_2,v_1,v_l)$;\label{2step2}
    \item for $i \in (N_1 \cup N_2)\setminus (N_1^* \cup N_2^*)$ the number of possibilities for $v_i$; \label{2step3}
    \item for $i\in N_1^*$, the number of possibilities for $(v_i,v_{i+1})$; \label{2step4}
    \item for $i\in N_2^*$, the number of possibilities for $(v_i,v_{i-1})$; \label{2step5}
    \item given the choices in \ref{2step2}, \ref{2step4}, and \ref{2step5}, the number of possibilities for vertices of the cycle 
not chosen in \ref{2step1}-\ref{2step5}. \label{2step6}
\end{enumerate}

For \ref{2step1} we just recall that Lemma \ref{degsize} bounds the number of choices for $v_k$ by 
\begin{equation*}\epsilon np^{1+\gamma}.
\end{equation*}

For \ref{2step2} the number of possibilities for $(v_2,v_1,v_l)$ is bounded by
\begin{align*}
    \sum\left\{\hat{d}(v_1)^2: np^{1-\gamma/2}\ge \hat{d}(v_1)>2np\right\}&\le \left(np^{1-\gamma/2}\right)\sum \{\hat{d}(v_1): \hat{d}(v_1)>2np\}\\
    &<^* \epsilon^2 n^3p^{3-\gamma/2}t,
\end{align*}
where the second inequality is given by Lemma \ref{degsum}.

For \ref{2step3}, Lemma \ref{degsize} bounds the number of choices for $v_i$ (of intermediate or large degree) by $\epsilon npt$; so the number of possibilities in \ref{2step3} is at most
\begin{equation*}
(\epsilon npt)^{|N_1|+|N_2|-|N_1^*|-|N_2^*|}.
\end{equation*} 

For $i$ as in \ref{2step4}, the number of possibilities for $(v_i,v_{i+1})$ is at most
\[
\sum\{ \hat{d}(v_i):\hat{d}(v_i)> 2np\}<^* \epsilon^2 n^2p^2t,
\]
with the inequality given by Lemma \ref{degsum}. Thus the number of possibilities in \ref{2step4} is at most
\begin{equation*}
    \left(\epsilon^2 n^2p^2t\right)^{|N_1^*|}.
\end{equation*}
Similarly, the total number of possibilities in \ref{2step5} is at most
\begin{equation*}
    \left(\epsilon^2 n^2p^2t\right)^{|N_2^*|}.
\end{equation*}

Finally, for \ref{2step6} we choose the remaining $v_i$'s with $i<k$ in increasing order (of their indices) and those with $i>k$ in decreasing order. In the first case, when we come to $v_i$ the number of possibilities is at most $\hat{d}(v_{i-1})\le 2np$ (since $v_{i-1} \notin N_1$), and similarly in the second case this number is at most $\hat{d}(v_{i+1})\le 2np$ since $v_{i+1}\notin N_2$.
Thus, the number of possibilities in \ref{2step6} is at most 
\begin{equation*}
    (2np)^{l-|N_1|-|N_2|-|N_1^*|-|N_2^*|-4}.
\end{equation*}
Combining the above bounds we find that, for a given $M$ and $N$, the number of cycles of type $(M,N)$ is at most
\begin{equation*}
    \epsilon^3n^lp^{l+\gamma/2}t^l2^l<\epsilon^3n^lp^l(10l^3)^l<\frac{\delta n^lp^l}{4l^23^{l}},
 \end{equation*}
 where the second inequality uses (\ref{ptbound}).
 
Now we assume $2 \in N_1 \cup \{k\}$, but $l \notin N_2 \cup \{k\}$. In this case \ref{2step1}, \ref{2step3}, \ref{2step4}, and \ref{2step5} and their respective bounds all remain the same. However, now we replace \ref{2step2} with
\begin{enumerate}[label=(\roman*$'$)]
\setcounter{enumi}{1}
    \item the number of possibilities for $(v_1,v_l)$.\label{3step2}
\end{enumerate}
This is because $v_2$ will be selected in either \ref{2step1}, \ref{2step3}, or \ref{2step4}.
Our new \ref{3step2} is bounded by
\begin{equation*}
    \sum\left\{\hat{d}(v_1):\hat{d}(v_1)>2np\right\}<^*\epsilon^2n^2p^2t,
\end{equation*}
where the inequality comes from Lemma \ref{degsum}.
Additionally, in \ref{2step6} there are now $l-|N_1|-|N_2|-|N_1^*|-|N_2^*|-3$ vertices left to choose. Thus our bound for \ref{2step6} becomes
\begin{equation*}
    (2np)^{l-|N_1|-|N_2|-|N_1^*|-|N_2^*|-3}.
\end{equation*}
Combining these bounds with our previous bounds for \ref{2step1} and \ref{2step3}-\ref{2step5} we find that, for a given $M$ and $N$, the number of cycles of type $(M,N)$ is at most
\begin{equation*}
     \epsilon^3n^lp^{l+\gamma}t^l2^l<\epsilon^3n^lp^l(4l^3)^l<\frac{\delta n^lp^l}{4l^23^{l}},
\end{equation*}
where the second bound is again given by (\ref{ptbound}).

The argument for $2 \notin N_1 \cup \{k\}$, $l\in N_2 \cup \{k\}$ is essentially identical to the preceding one, so we will not discuss it further. 

It remains to consider the case when we have both $2 \in N_1 \cup \{k\}$ and $l \in N_2\cup \{k\}$. Again, there is no change in \ref{2step1} and \ref{2step3}-\ref{2step5} and we replace \ref{2step2}, in this case, by 
\begin{enumerate}[label=(\roman*$''$)]
\setcounter{enumi}{1}
    \item the number of possibilities for $v_1$\label{5step2}
\end{enumerate}
(since $v_2$ and $v_l$ will be among the vertices chosen in \ref{2step1} and \ref{2step3}-\ref{2step5}). By Lemma \ref{degsize} the number of possibilities here (i.e.\ for $v_1$) is at most
\begin{equation*}
    \epsilon^2npt.
\end{equation*}
Additionally, in \ref{2step6} we are now selecting $l-|N_1|-|N_2|-|N_1^*|-|N_2^*|-2$ vertices; so, our bound becomes
\begin{equation*}
    (2np)^{l-|N_1|-|N_2|-|N_1^*|-|N_2^*|-2}.
\end{equation*}
Again, combining bounds, we find that the number of cycles of type $(M,N)$ is at most
\begin{equation*}
     \epsilon^3n^lp^{l+\gamma}t^l2^l<\epsilon^3n^lp^l(4l^3)^l<\frac{\delta n^lp^l}{4l^23^{l}}.
\end{equation*}

So to recap, we have shown that, for any given $M$, $N \neq \emptyset$ (where we assume $\hat{d}(v_k)>np^{1-\gamma}$ and $np^{1-\gamma/2}\ge \hat{d}(v_1)>2np$) there are at most 
\begin{equation*}
    \frac{\delta n^lp^l}{4l^23^{l}}
\end{equation*}
cycles of type $(M,N)$. 

Since there are fewer than $3^l$ choices for $(M,N)$ and the assumptions on $1$ and $k$ only cost a factor of $l^2$, there are at most
\begin{equation*}
    \frac{\delta n^lp^l}{4}
\end{equation*}
cycles of all types $(M,N)$ with $N \neq \emptyset$; recalling (see (\ref{Nempty})) that we showed the same bound for the number of cycles of types $(M, \emptyset)$ (with $M \neq \emptyset$), we have the desired bound, $(\delta/2) n^lp^l$, on the l.h.s.\ of (\ref{case1}).

\end{proof}

\section{Proof of (\ref{pathtotal})}\label{pf3}
\begin{proof}[\unskip\nopunct]
For the rest of our discussion we may ignore \emph{bad} vertices, meaning those of degree at least $np^{1-\gamma}$, since cycles involving such vertices are excluded from (\ref{pathtotal}). (Recall we are calling $\hat{d}(v)$ the degree of $v$.)

What's really going on here is as follows. We think of choosing $\nabla(V_1, V_l)$ after all other edges have been specified. The number of cycles (again, avoiding bad vertices) is then
\begin{equation}\label{case3total}
\sum_{v_1 \sim v_l}f(v_1, v_l)
\end{equation}
(recall $f(v_1, v_l)$ is the number of full paths with endpoints $v_1$ and $v_l$ in which there are no bad vertices). Given $G \setminus \nabla(V_1, V_l)$, this is a weighted sum of independent binomials with expectation 
\begin{equation}\label{case3exp}
    p\sum_{v_1, v_l}f(v_1, v_l),
\end{equation}
to which we may hope to apply the large deviation bound in Lemma \ref{sum}. In this section we give a good (w.l.p.) bound on the sum in (\ref{case3exp}) (namely (\ref{pathtotal})). Once we have this, the only difficulty is that some of the ``weights'' $f(v_1, v_l)$ may be too large to support finishing \emph{via} the lemma. We will handle this difficulty in Section \ref{pf2}. 

\medskip

To prove (\ref{pathtotal}) we first consider full paths $(v_1, \ldots, v_l)$ in which each of $v_1, \ldots$ $,v_{l-1}$ has degree at most $(1+\epsilon)np$. There are at most 
\begin{equation}\label{bigpath}
    (1+\epsilon)^ln^lp^{l-1}<(1+\delta/16)n^lp^{l-1}
\end{equation}
such paths.

Now all the paths $(v_1, \ldots, v_l)$ left to consider must have some $v_i$ (where $i \in [l-1]$) such that $\hat{d}(v_i)> (1+\epsilon)np$. 
To count the number of such paths we split the argument based on $p$. First assume 
\begin{equation}\label{pupper3}
    p> \frac{\ln^2n}{n}.
\end{equation}
(This is not a tight bound for either argument, but it is a convenient cut-off.) Given (\ref{pupper3}) we know 
\begin{equation*} 
    q_K \le \exp\left[\frac{-\epsilon^2Knp}{16}\right]< \exp\left[\frac{-\epsilon^2\ln^2n}{16}\right]<n^{-2} 
\end{equation*}
for all $K\ge 1+\epsilon$ (see (\ref{qk}) for the definition of $q_K$), so in applications of Lemma \ref{degsize} we are always using the second value of $r_K$ (namely, $r_K=\frac{\epsilon^2 npt}{K\ln K}$). Additionally since $p>\frac{\ln^2n}{n}$ Lemma \ref{degsum} applies.
As in Section \ref{pf1} we classify paths according to the positions of vertices with $\hat{d}(v_i)>(1+\epsilon)np$. For $M \subseteq [l-1]$, say $v_i$ is of type $M$ if 
\begin{equation*}
   \hat{d}(v_i)
    \begin{cases}
    >(1+\epsilon)np,  &\mbox{if $i \in M$,}\\ 
    \le (1+\epsilon)np  &\mbox{otherwise,}\\ 
    \end{cases}
\end{equation*}
and say a set of vertices is of type $M$ if each of its members is either of type $M$ or in $V_l$. Note we have already shown that there are at most
\begin{equation*}
    (1+\delta/16)n^lp^{l-1}
\end{equation*} full paths of type $\emptyset$, so we now assume $M \neq \emptyset$. Let $m$ be the smallest element of $M$ and let 
\begin{equation*}
    M^*= \{i \in M : i+1 \notin M \}.
\end{equation*}
We will bound:
\begin{enumerate}[label=(\roman*)]
    \item for $i \in M \setminus M^*$, the number of possibilities for $v_i$;\label{case3step1}
    \item for $i \in M^*$, the number of possibilities for $(v_i, v_{i+1})$;\label{case3step2}
    \item given the choices in \ref{case3step2}, the number of possibilities for vertices of the path not chosen in \ref{case3step1} and \ref{case3step2}. \label{case3step3}
\end{enumerate}
For i as in \ref{case3step1} we recall that by Lemma \ref{degsize} the number of $v_i$'s of degree at least $(1+\epsilon)np$ is at most $\epsilon npt$. 
So, the total number of possibilities in \ref{case3step1} is at most
\begin{equation}\label{case3step1bound}
    (\epsilon npt)^{|M|-|M^*|}.
\end{equation}
For $i$ as in \ref{case3step2}, the number of possibilities for $(v_i,v_{i+1})$ is at most
\begin{equation*}
\sum\left\{\hat{d}(v_i):\hat{d}(v_i)> (1+\epsilon)np\right\}<^* \epsilon^2 n^2p^2t,
\end{equation*}
with the inequality given by Lemma \ref{degsum}. Thus the total number of possibilities in \ref{case3step2} is at most
\begin{equation}\label{case3step2bound}
    (\epsilon^2 n^2p^2t)^{|M^*|}.
\end{equation}
Finally for \ref{case3step3} we choose the remaining $v_i$'s with $i >m$ in increasing order (of the indices). When we come to $v_i$ we know $i-1 \notin M$, so given $v_{i-1}$ there are at most $(1+\epsilon)np$ choices for $v_i$. If $m=1$ then we have selected all the vertices in the path. If not, then we next select $v_{m-1}$. Since we are ignoring vertices of degree at least $np^{1-\gamma}$ we know that given $v_m$ there are at most $np^{1-\gamma}$ ways to select $v_{m-1}$. If $m=2$ then we are done, and if not then we select the $v_i$'s with $i<m-1$ in decreasing order (of the indices). Since $i+1 \notin M$, given $v_{i+1}$ there are at most $(1+\epsilon)np$ choices for $v_i$. 
Thus, the number of possibilities in \ref{case3step3} is at most
\begin{equation}\label{case3step3bound}
    \begin{cases}
    ((1+\epsilon)np)^{l-|M|-|M^*|-1}(np^{1-\gamma})  &\mbox{if $m>1$,}\\ 
    ((1+\epsilon)np)^{l-|M|-|M^*|}  &\mbox{if $m=1$.}
    \end{cases}
\end{equation}
Combining (\ref{case3step1bound}), (\ref{case3step2bound}), and the appropriate bound from (\ref{case3step3bound}) we find that, for a given $M$, there are at most
\begin{equation*}
\epsilon (1+\epsilon)^ln^lp^{l-\gamma}t^l<\epsilon (2l)^ln^lp^{l-1}<\frac{\delta n^lp^{l-1}}{2^{l+3}}
\end{equation*}
full paths of type $M$ (where the first inequality uses (\ref{ptbound})). Since there are less than $2^{l-1}$ possibilities for $M \neq \emptyset$ there are at most
\begin{equation*}
    \frac{\delta n^lp^{l-1}}{16}
\end{equation*}
full paths of type other than $\emptyset$. Together with our earlier bound on the number of full paths of type $\emptyset$ this bounds the total number of full paths (without vertices of degree at least $np^{1-\gamma}$) by
\begin{equation*}
    \left(1+\delta/8\right) n^lp^{l-1},
\end{equation*}
as desired. 

When 
\begin{equation}\label{plower3}
p\le \frac{\ln^2 n}{n}
\end{equation}
we first note that we have a better bound on $\Delta$ (the maximum degree) than $np^{1-\gamma}$. For (\ref{plower3}) Lemma \ref{basic} with $K = (\ln^3n)/2$ (and $x$ any vertex) gives
\begin{align*}
    \pr(\Delta>\ln^3n (np))&\le ln \pr(d(x)>\ln^3n(np)) \\
    &<ln\exp[-np(\ln^3 n)(\ln \ln n)]\\
    &<\exp[-\Omega_{\delta,l}(n^2p^2t)],
\end{align*}
using $npt<\ln^3n$ and absorbing the initial $ln$ into the exponent (since (\ref{plower}) gives $np(\ln^3n)>\epsilon^{-2}(\ln^3n)$).
Thus, $\Delta <^* \ln^3n(np)\le\ln^5 n$. 

Given $p$, let $K$ be minimal with $q_K \le n^{-2}$. We first bound the number of cycles containing at least one $v$ with $\hat{d}(v)>Knp$. Lemma \ref{degsize} says there are at most $\frac{l\epsilon^2npt}{K \ln K}$ such vertices (in all of $V$).
Once such a vertex $v$ has been specified there are at most
\begin{equation*}
    \Delta^{l-1}<^*\ln^{5(l-1)}n
\end{equation*}
ways to select the remaining vertices in a full path containing $v$. So, w.l.p.\ we have at most
\begin{equation}\label{largeipathbound}
    \frac{l\epsilon^2npt\ln^{5(l-1)}n}{K \ln K}=o(n^lp^{l-1})
\end{equation}
full paths containing at least one $v$ as above. (The quite weak $o(n^lp^{l-1})$ follows from the lower and upper bounds on $p$ in (\ref{plower}) and (\ref{plower3}), respectively.)

Now we count paths in which every vertex has degree at most $ Knp$ and at least one vertex has degree at least $(1+\epsilon)np$ (recalling that we have already treated those violating either condition).
Say $v$ is of type $i$ if 
\begin{equation*}
    (1+\epsilon)2^inp< \hat{d}(v)\le(1+\epsilon)2^{i+1}np,
\end{equation*}
and let $U_i= \{\text{vertices of type }i\}$. We say the type of a path $P$ is the largest $i$ for which $P$ contains a vertex of type $i$.
Lemma \ref{degsize} gives
\begin{equation*}
    |U_i|<^*6l\epsilon2^{-il} n.
\end{equation*}
Note we have already bounded the number of full paths of type $i$ where $i>\log_2K-1$.
For smaller $i$ we think of specifying a path $P$ of type $i$ by choosing
\begin{enumerate}[label=(\roman*)]
    \item some $v$ of type $i$, and then \label{case3lowerp1}
    \item the remaining vertices of the path. \label{case3lowerp2}
\end{enumerate}
Here the bounds are easy: the number of possibilities in \ref{case3lowerp1} is at most
\begin{equation}
    |U_i|<^*6l\epsilon2^{-il} n,
\end{equation}
and the number of possibilities in \ref{case3lowerp2} is at most
\begin{equation*}
    ((1+\epsilon)2^{i+1}np)^{l-1},
\end{equation*}
since, given the choice in \ref{case3lowerp1}, we may order the remaining choices so that each new vertex is drawn from the at most $(1+\epsilon)2^{i+1}np$ neighbors of some vertex chosen earlier. Thus the number of full paths of type $i$ is bounded by
\begin{equation*}
    6l\epsilon(1+\epsilon)^{l-1}2^{l-i-1}n^lp^{l-1}<\epsilon l 2^{2l-i}n^lp^{l-1}.
\end{equation*}
Summing over $i$ we find that w.l.p.\ there are at most
\begin{equation}\label{smallipathbound}
    \sum_{i=0}^{\log_2K-1}\epsilon l 2^{2l-i}n^lp^{l-1}<\frac{\delta}{17}n^lp^{l-1}
\end{equation}
full paths of all types up to $\log_2K-1$ (where the inequality follows easily from our choice of $\epsilon$ --- see (\ref{epsilon})).
Adding (\ref{smallipathbound}) to the numbers of full paths with all degrees at most $(1+\epsilon)np$ and those of type $i$ for $i>\log_2K-1$ ((\ref{bigpath}) and (\ref{largeipathbound})) we find that w.l.p.\ there are at most
\begin{equation*}
    (1+\delta/8)n^lp^{l-1}
\end{equation*}
full paths (with all vertices of degree at most $np^{1-\gamma}$).
So, regardless of $p$, we have 
\begin{equation*}
    \sum f(v_1,v_l)<^*(1+\delta/8)n^lp^{l-1},
\end{equation*}
as desired.

\end{proof}


\section{Proof of (\ref{case2})}\label{pf2}
\begin{proof}[\unskip\nopunct]
As explained at the start of Section \ref{pf3} we want to use (\ref{pathtotal}) and finish via Lemma \ref{sum}, but some $f(v_1,v_l)$'s may be too large to support this. To handle this difficulty we introduce the notion of a ``heavy path'' below. We then set 
\begin{equation*}
\C'= \{(v_1, \ldots, v_l) \in \C: (\forall i)v_i \in V'_i \text{ and }   (v_1,\ldots,v_l) \text{ is not heavy}\},
\end{equation*}
and show
\begin{equation} \label{case2a}
         \pr(|\C'|>(1+\delta/4)n^lp^l)< \exp[- \Omega_{\delta,l}(n^2p^2s)]\text{, and}
\end{equation} 
\begin{equation} \label{case2b}
    \text{w.l.p. } |\{(v_1, \ldots, v_l) \in \C: \forall i (v_i \in V'_i ), (v_1,\ldots,v_l) \text{ heavy}\}|< (\delta/4) n^lp^l.
\end{equation}

It will turn out that we need different definitions of ``heavy path'', depending on $p$. Either of these will say that the number of non-heavy paths, say $g(v_1,v_l)$, joining any $v_l,v_l$ satisfies
\begin{equation}\label{nonheavynumber}
    g(v_1,v_l)\le \frac{4^ln^{l-2}p^{l-2}}{s}.
\end{equation}
(Recall $s= \min\{t,n^{l-2}p^{l-2}\}$.)
We will return to the definitions of heavy path and the proof of (\ref{case2b}) in Subsections \ref{case2bhigh} and \ref{case2blow}; here we assume (\ref{nonheavynumber}) and give the easy proof of (\ref{case2a}). 

As suggested above this is a straightforward application of Lemma \ref{sum}. Let $V_1 = \{x_1, \ldots, x_n\}$ and $V_l= \{y_1, \ldots, y_n\}$. Then with 
\begin{equation*}
    w_{i,j}= g(x_i,y_j)\le \frac{4^ln^{l-2}p^{l-2}}{s} \equalscolon z
\end{equation*}
and $\zeta_{i,j}$ the indicator of the event $\{x_iy_j \in \Hb\}$ we have 
\begin{equation*}
    |\C'|= \zeta \coloneqq \sum \zeta_{i_j}w_{i,j}.
\end{equation*}
In addition, recalling (\ref{pathtotal}), we have
\begin{equation*}
    \E \zeta = p \sum w_{i,j} \le p \sum f(v_1, v_l)<^* (1+\delta/8)n^lp^l.
\end{equation*}
Hence Lemma \ref{sum} with $\lambda = (\delta/8) n^lp^l$ gives $$\pr(|\C'|> (1+\delta/4)n^lp^l)< \exp[- \Omega_{\delta,l}(n^2p^2s)],$$
as desired. 


\subsection{Proof of (\ref{case2b}) when $p> n^{\frac{-5l}{5l+1}}$}\label{case2bhigh}
For $p> n^{\frac{-5l}{5l+1}}$ we say $(v_1, v_l)$ is \emph{heavy} if 
\begin{equation*}
f(v_1,v_l)> \frac{4^ln^{l-2}p^{l-2}}{s},
\end{equation*} 
and $(v_1,\ldots, v_l)$ is a \emph{heavy} path if $(v_1,v_l)$ is heavy.
(Note that here we have $s=t(=\ln (1/p))$.)
So, in this case the notion of heavy depends only on the endpoints of the path.
Note that this definition trivially implies (\ref{nonheavynumber}).

A brief indication of why we need two definitions of a heavy path may be helpful. In the present case (i.e.\ $p>n^{\frac{-5l}{5l+1}}$) we bound the number of cycles $(v_1, \ldots, v_l)$ for which $(v_1, \ldots, v_l)$ is a heavy path by first bounding the number of $v_1$'s (and similarly $v_l$'s) that are in heavy paths. To do this we show that for $v_1$ to be in a heavy path there must be some $v_3$ for which $d(v_1,v_3)(\coloneqq |N(v_1)\cap N(v_3)|)$ is ``large'', and we use this necessary condition to bound the number of $v_1$'s in heavy paths.

\medskip

Let 
\begin{align*}
V^*_1&= \{v_1 \in V'_1 : \exists v_l \in V_l \text{ with } (v_1,v_l) \text{ heavy}\},\\
V^*_l&= \{v_l \in V'_l : \exists v_1 \in V_1 \text{ with } (v_1,v_l) \text{ heavy}\}.
\end{align*}
Thus every cycle, $(v_1, \ldots, v_l)$, considered in this section must have $v_1 \in V_1^*$ and $v_l \in V_l^*$. We first bound $|V_1^*|$ and $|V_l^*|$, and then use this to bound $|\nabla(V_1^*, V_l^*)|$. A necessary condition for $v_1 \in V_1^*$ is
\begin{equation}\label{starcond}
    \text{ there exists }v_3 \text{ such that }d(v_1,v_3)\ge np^{1+\gamma(l-1)}.
\end{equation}
To see this, fix $v_1$ and recall that $\hat{d}(v)<np^{1-\gamma}$ for every vertex under discussion in (\ref{case2}). Thus, we know that for any $v_l$ there are at most $(np^{1-\gamma})^{l-3}$ paths $(v_l,\ldots, v_3)$. To pick $v_2$ to complete such a path with $v_1$ we require $v_2 \in N(v_1) \cap N(v_3)$. Thus if $d(v_1,v_3)<np^{1+\gamma(l-1)}$ for all $v_3$ then for any $v_l$,
\begin{equation*}
    f(v_1,v_l) < n^{l-2}p^{l-2+2\gamma}< \frac{5l^2n^{l-2}p^{l-2}}{2et} <4^ln^{l-2}p^{l-2}/s.
\end{equation*}
(Here the middle inequality comes from (\ref{ptbound}) with $\beta=2\gamma$ and $k=1$.)
So in order to bound $|V_1^*|$ it suffices to bound the number of $v_1$'s satisfying (\ref{starcond}).

Since $\hat{d}(v_3)< np^{1-\gamma}$, Lemma \ref{basic} (with $m= np^{1-\gamma}$, $\alpha=p$, and $K=p^{-1+\gamma l}$) gives
\begin{align*}
    \pr(v_1 \in V_1^*)&\le n\pr(B(m,p)>Kmp) \\ 
    &<n\exp\left[np^{1+\gamma(l-1)}(1-(1-\gamma l)t)\right].
\end{align*}
Note that $p\le e^{-4/\gamma}$ (see (\ref{pupper})) implies $t \ge 4/\gamma$, so 
\begin{equation*}
 \exp[np^{1+\gamma(l-1)}(1-(1-\gamma l)t)]<\exp[-np^{1+\gamma(l-1)}t/2].   
\end{equation*}
Thus,
\begin{align*}
    \pr(v_1 \in V_1^*)&<n\exp[-np^{1+\gamma(l-1)}t/2]\\&< \exp[-np^{1+\gamma(l-1)}t/3].
\end{align*} The initial $n$ disappears since $p> n^{\frac{-5l}{5l+1}}$ implies $np^{1+\gamma(l-1)}>n^{1/(5l^2+l)}$.

Next we show that w.l.p.\ $|V_1^*|$ and $|V_l^*|$ are at most $\epsilon np^{1-\gamma(l-1)}$. The lemma will be stated in more generality as we will use it again after (\ref{vistarprob}).
\begin{lemma}\label{V*bound}
If $c \in [1,3]$ and $U$ is a random subset of $V_i$ in which each $v_i$ is included independently with probability at most $\exp[-np^{1+\gamma(l-c)}t/3]$ then $|U|<^* \epsilon np^{1-\gamma(l-c)}$.
\end{lemma}
\begin{subproof}
Here we apply Lemma \ref{basic} with $m=n$, $\alpha = \exp[-np^{1+\gamma(l-c)}/3]$ and $K=\epsilon p^{1-\gamma(l-c)}\alpha^{-1}$. Note that since $p\ge n^{\frac{-5l}{5l+1}}$ we know, say, $K/e>\alpha^{-1/2}$; so Lemma \ref{basic} gives
\begin{align*}
    \pr\left(|U|>np^{1-\gamma(l-c)}\right)&<(e/K)^{\epsilon np^{1-\gamma(l-c)}}\\
    &<\alpha^{\epsilon np^{1-\gamma (l-c)}/2}\\
    &=\exp[-\epsilon n^2p^2t/6].
\end{align*}
\end{subproof}

\noindent Hence $|V_1^*|,|V_l^*|<^*\epsilon np^{1-\gamma(l-1)}$.

We next show that for any $i$
\begin{align}\label{ABbound}
 &\text{w.l.p.\ }|\nabla(A,B)|< \epsilon^2 n^2p^2 \\ &\forall A \subseteq V_i, B \subseteq V_{i+1} \text{ with }|A|,|B|< \epsilon np^{1-\gamma(l-1)}. \nonumber
\end{align}

We use (\ref{ABbound}) to bound $|\nabla(V_1^*,V_l^*)|$ (and again after (\ref{vistarprob})). To prove (\ref{ABbound}) we assume $A$ and $B$ are of the appropriate sizes and apply Lemma \ref{basic} with $m=|A||B|$, $\alpha=p$, and $K=\epsilon^2n^2p^2(mp)^{-1}$. Note that $m<\epsilon^2 n^2p^{2-2\gamma(l-1)}$, and, generously, $K \ge p^{-1+[2/(5l)]}>p^{-1/2}$. Also, since $p\le e^{-20l^2}$, we have $\ln(K)>t/2 \ge 10l^2$. So for a given $A$ and $B$ of the appropriate size Lemma \ref{basic} gives
\begin{align*}
    \P(|\nabla(A,B)|> \epsilon^2 n^2p^2)&<\exp[-\epsilon^2 n^2p^2(\ln(K)-1)]\\
    &<\exp[-\epsilon^2 n^2p^2t/4].
\end{align*}
Simply taking the union bound with the first sum over all possible $A, B$ and the next two over all $a,b<\epsilon np^{1-\gamma(l-1)}$ we have

\begin{align}
&\sum_{A,B}\P(|\nabla(A,B)|> \epsilon^2 n^2p^2)<\nonumber \\
&\sum_{a,b} l{n \choose a}{n \choose b}\exp[-\epsilon^2 n^2p^2t/4]< \nonumber \\
&\sum_{a,b}l\exp[ a \ln(en/a)+b\ln(en/b)-\epsilon^2 n^2p^2t/4].\label{nabla}
\end{align}
It is easy to see (using $p>n^{\frac{-5l}{5l+1}}$ and $\gamma = \frac{1}{5l^2}$) that for $a,b <\epsilon np^{1-\gamma(l-1)}$ we have 
\begin{equation*}
    n^2p^2t \gg \max\{a \ln(en/a)+b\ln(en/b),\ln(n)\}.
\end{equation*}
So (\ref{nabla}) is, for example, at most $\exp[-\epsilon^2 n^2p^2t/5]$. 
Therefore w.l.p.\
\begin{equation}\label{ABnabla}
|\nabla(A,B)|< \epsilon^2 n^2p^2, \text{ for all } A,B \text{ with } |A|,|B|<np^{1-\gamma(l-1)},
\end{equation} 
as desired. 
Specifically we have (w.l.p.)
\begin{equation}\label{V1Vlstarnabla}
    |\nabla(V_1^*, V_l^*)|<\epsilon^2 n^2p^2.
\end{equation}

We next want to bound the number of full paths between $V_1^*$ and $V_l^*$. For $i \in \{2, \ldots, l-1\}$ let
\begin{equation*}
    V_i^*= \{v_i:\max_{v \in V_{i-2}\cup V_{i+2}}d(v,v_i) > np^{1+\gamma (l-3)}\}.
\end{equation*}
We first bound the number of full paths such that at least one vertex $v_i$ in the path is not in the appropriate $V^*_i$. Fixing  $v_1$, $v_l$, and an index $i<l-1$ we bound the number of full paths $(v_1, \ldots, v_l)$ with $v_i \notin V_i^*$. Since $\hat{d}(v)<np^{1-\gamma}$ for all $v$ under consideration, there are at most 
\begin{equation*}
    n^{i-1}p^{(1-\gamma)(i-1)}
\end{equation*}
ways to choose $v_2\sim  \cdots\sim v_i$ with $v_2 \sim v_1$ and 
\begin{equation*}
n^{l-i-2}p^{(1-\gamma)(l-i-2)}
\end{equation*} ways to choose $v_{l-1}\sim \cdots\sim v_{i+2}$ with $v_{l-1}\sim v_l$. To complete the path we must have $v_{i+1} \in N(v_i) \cap N(v_{i+2})$. Since we assume $v_i \notin V^*_i$, there are at most $np^{1+\gamma(l-3)}$ choices for $v_{i+1}$. Thus there are at most 
\begin{equation*}
(n^{i-1}p^{(1-\gamma)(i-1)})(n^{l-i-2}p^{(1-\gamma)(l-i-2)})np^{1+\gamma(l-3)}= n^{l-2}p^{l-2}
\end{equation*} paths from $v_1$ to $v_l$ with $v_i \notin V^*_i$. 

If $i=l-1$ then we instead bound the number of choices for $v_{l-1}$ by 
\begin{equation*}
    \hat{d}(v_l)<np^{1-\gamma},
\end{equation*}
and the number of ways to choose $v_2 \sim \cdots \sim v_{l-3}$ with $v_2 \sim v_1$ by
\begin{equation*}
    n^{l-4}p^{(l-4)(1-\gamma)}.
\end{equation*}
To complete the path we must have $v_{l-2}\in N(v_{l-3})\cap N(v_{l-1})$. Again, as we are assuming $v_{l-1}\notin V^*_{l-1}$, there are at most $np^{1+\gamma(l-3)}$ choices for $v_{l-2}$. So, there are at most 
\begin{equation*}
(np^{1-\gamma})(n^{l-4}p^{(1-\gamma)(l-4)})(np^{1+\gamma(l-3)})= n^{l-2}p^{l-2}
\end{equation*} paths from $v_1$ to $v_l$ with $v_{l-1} \notin V^*_{l-1}$.

Now summing over $i$, there are at most $(l-2)n^{l-2}p^{l-2}$ paths using at least one vertex outside of $\bigcup_{i=2}^{l-2} V^*_i$, and combining this with (\ref{V1Vlstarnabla}) bounds the number of cycles as in (\ref{case2b}) (with some vertex outside of $\bigcup_{i=2}^{l-2} V^*_i$) by
\begin{equation}\label{nonstarbound}
  (l-2)\epsilon n^{l}p^{l} < \frac{\delta}{8} n^lp^l.
\end{equation}

The only cycles left to count are those with $v_i \in V_i^*$ for all $i$. We first bound $|V_i^*|$. Lemma \ref{basic} with $m=np^{1-\gamma}$, $\alpha=p$, and $K=p^{-1+\gamma(l-2)}$ (and the union bound) gives, for any $v \in V_i$,
\begin{align}
    \pr(v \in V_i^*)&< 2n\pr(B(m,p)>Kmp)\nonumber\\
    &<2n\exp[np^{1+\gamma(l-3)}(1-(1-\gamma (l-2))t)].\label{vistarbound}
\end{align}
As before, $ t \ge 4/\gamma$ implies the r.h.s.\ of (\ref{vistarbound}) is at most
\begin{equation*}
 2n\exp[-np^{1+\gamma(l-3)}t/2].   
\end{equation*}
Hence, 
\begin{align}\label{vistarprob}
    \pr(v_i \in V_i^*)&<2n\exp[- np^{1+\gamma(l-3)}t/2] \nonumber \\&< \exp[-np^{1+\gamma(l-3)}t/3].
\end{align} Again the initial $2n$ disappears since $p> n^{\frac{-5l}{5l+1}}$ implies $np^{1+\gamma(l-3)}>n^{3/(5l^2+l)}$.
Given (\ref{vistarprob}) Lemma \ref{V*bound} gives $|V_i^*|<^*\epsilon np^{1-\gamma(l-3)}$. Assuming this, (\ref{ABnabla}) gives
\begin{equation*}|\nabla(V_i^*, V_{i+1}^*)|< \epsilon n^2p^2
\end{equation*} for all $i$. 

To finish the proof (for $p\ge n^{\frac{-5l}{5l+1}}$) we use the following lemma due to Shearer \cite{S}. We will use this lemma again when $p \le n^{\frac{-5l}{5l+1}}$. To state it we require the following definition. (Recall a hypergraph on $V$ is simply a collection --- possibly with repeats --- of subsets of $V$.)

For a hypergraph $\F$ on the vertex set $V$ and $H \subseteq V$, the \emph{trace} of $\F$ on $V$ is defined to be
\begin{equation*}
  \Tr(\F, H)= \{F \cap H: F \in \F\}.
\end{equation*}

\begin{lemma}\label{shearer}
Suppose $\F$ is a hypergraph on $V$ and $\HH$ is another hypergraph on $V$ such that every vertex in $V$ belongs to at least $d$ edges of $\HH$. Then 
\begin{equation}\label{shearereq}
    |\F|\le \prod_{H \in \HH} |\Tr(\F, H)|^{1/d}.
\end{equation}
\end{lemma}

To apply Lemma \ref{shearer} here, let $\F$ be the hypergraph on $V=V(\Hb)$ whose edges are the vertex sets of cycles using only vertices in $\bigcup_{i=1}^l V^*_i$. So $|\F|$ is the number of cycles using only vertices in $\bigcup_{i=1}^l V_i^* $. Let $\HH$ be the hypergraph on $V$ with edges $\{H_i \coloneqq V_i \cup V_{i+1}\}_{i\in [l]} $. Thus each vertex belongs to exactly two edges of $\HH$. Furthermore 
\begin{equation*}
|\Tr(\F, H_i)|\le |\nabla(V_i^*,V_{i+1}^*)|<\epsilon n^2p^2. 
\end{equation*}
Thus Lemma \ref{shearer} gives
\begin{equation*}
    |\F|\le \prod_{H \in \HH}|\Tr(\F,H)|^{1/2}<(\epsilon n^2p^2)^{l/2}<(\delta/8) n^lp^l.
\end{equation*}
Combining this with (\ref{nonstarbound}) gives (\ref{case2b}) (for $p>n^{\frac{-5l}{5l+1}}$).

\subsection{Proof of (\ref{case2b}) when $p\le n^{\frac{-5l}{5l+1}}$}\label{case2blow}
For $p\le n^{\frac{-5l}{5l+1}}$ we need the following definitions for $j \notin \{1,l\}$ and $i<l-1$
\begin{align}
    N^j(v_l)= &\{v_j : \text{ there exists a path } (v_j,v_{j+1},\ldots, v_l)\} \label{lowpheavy} \\
    V''_i = &\{v_i \in V_i : \max_{v_l}d_{N^{i+1}(v_l)}(v_i) >4\}. \label{lowpheavy2}
\end{align}
That is, $v_i \in V_i''$ if, for some $v_l$, $v_i$ has at least 5 neighbors in $V_{i+1}$ that are  ``directly reachable'' from $v_l$.
We say a path $(v_1, \ldots, v_l)$ is \emph{heavy} if $v_i \in V''_i$ for some $i(<l-1)$. Note (as promised) we still have (\ref{nonheavynumber}), since 
\begin{equation*}
g(v_1,v_l)\le 4^{l-2}<\frac{4^l n^{l-2}p^{l-2}}{s}.
\end{equation*}
(Again recall $s= \min\{t, n^{l-2}p^{l-2}\}$.)

In this section we are bounding the number of cycles $(v_1, \ldots,v_l)$ containing at least one vertex in some $V''_i$. To do this we fix $i$ and bound the number of cycles with $v_i \in V''_i$. 

We first observe that
\begin{equation}\label{deltabound2}
    \Delta<^*n^2p^2t
\end{equation}
(where, as usual, $\Delta$ is the maximim degree in $\Hb$.)
For (\ref{deltabound2}) Lemma \ref{basic} with $K=npt/2$ (and $x$ any vertex), together with the union bound, gives
\begin{align*}
    \pr(\Delta>n^2p^2t)&\le ln \pr(d(x)>2Knp)\\
    &<ln\exp[-2Knp(1-\ln(K))]\\
    &<\exp[-n^2p^2t].
\end{align*}
So we may assume $\Delta<n^2p^2t$, whence, for any $j$ and $v_l$,
\begin{equation}\label{Njbound}
|N^j(v_l)| \le \Delta^{l-2}<n^{2l-2}p^{2l-2}t^{l-1}\equalscolon m.
\end{equation}
Note that $m \le n^{\frac{2l-2}{5l+1}}\log^{l-1} n $ (since $p\le n^{\frac{-5l}{5l+1}}$).

We next show
\begin{equation}\label{V''bound}
    |V''_i|<^* \epsilon n^2p^2.
\end{equation}
Here, for a given $v_l$, we may think of $N^{i+1}(v_l)$ --- which does not depend on edges involving $V_i$ --- as given. Then for a given $v_i$ we have (using (\ref{Njbound}))
\begin{equation}\label{V''prob}
    \pr(v_i \in V''_i) <n\pr(B(m,p)>4);
\end{equation}
so applying Lemma \ref{basic} with $\alpha =p$ and $K = 4m^{-1}p^{-1}>n^{3/5}$ bounds the r.h.s.\ of (\ref{V''prob}) by 
\begin{equation*}
    n(e/K)^4<e^4n^{-7/5}\equalscolon q.
\end{equation*}

Another application of Lemma \ref{basic}, with $m=n$, $\alpha =p$, and $K = \epsilon np^2q^{-1}> n^{2/5}$ now gives (\ref{V''bound}):
\begin{equation*}
    \pr(|V''_i|> \epsilon n^2p^2)< (e/K)^{\epsilon n^2p^2}<\exp[-(\epsilon/5) n^2p^2t].
\end{equation*}
We may thus assume from now on that $|V_i''|<\epsilon n^2p^2$.

Given $V''_i$ we bound the number of cycles $(v_1, \ldots, v_i, \ldots, v_l)$ with $v_i \in V''_i$. This requires the following definitions (for $i \neq j$):
\begin{align*}
V_{i,j}^0 &= \{v_j: \text{there is a path } (v_i, v_{i+1}, \ldots, v_j) \text{ with }v_i \in V''_i \},\\
V_{i,j}^1 &= \{v_j: \text{there is a path } (v_i, v_{i-1}, \ldots, v_j) \text{ with }v_i \in V''_i\},\\
V_{i,j}&= V_{i,j}^0 \cap V_{i,j}^1.
\end{align*}
(Note we are reading subscripts$\Mod l$.)

Thus $v_i \in V_{i,j}$ if and only if some cycle containing $v_j$ meets $V_i''$. We also set
\begin{equation*}
    V_{i,i}=V_{i,i}^0 = V_{i,i}^1=V''_i.
\end{equation*} 

To bound the number of cycles involving some $v_i \in V''_i$ we need a bound on $|\nabla(V_{i,j}, V_{i,{j+1}})|$, but will actually bound the (larger) quantity
\begin{equation*}
    |\nabla(V^0_{i,j}, V^1_{i,{j+1}})|.
\end{equation*}

As elsewhere the point here is to retain some independence; \emph{given} $V_i''$, $V_{i,j}^0$ and $V_{i,j+1}^1$ do not depend on $\nabla(V^0_{i,j}, V^1_{i,{j+1}})$. Thus, having specified $V_i''$ we may think of first exposing the edges of $\Hb$ not involving $\nabla(V^0_{i,j}, V^1_{i,{j+1}})$ --- thus determining $V_{i,j}^0$ and $V_{i,j+1}^1$ --- at which point $\nabla(V^0_{i,j}, V^1_{i,{j+1}})$ is just a binomial to which we may apply Lemma \ref{basic}. Note, however, that $\nabla(V^0_{i,j}, V^1_{i,{j+1}})$ will not be independent of the choice of $V_i''$, so we will need to take a union bound over possibilities for $V_i''$.

We will show 
\begin{equation}\label{nablabound}
|\nabla(V_{i,j}^0,V_{i,j+1}^1)|<^* \left(\frac{\delta}{4l}\right)^{2/l}n^2p^2.
\end{equation}
The eventual punchline here will be an application of Lemma \ref{shearer} (Shearer's Lemma) similar to the one in Section \ref{case2bhigh}. This is the reason for the $\left(\frac{\delta}{4l}\right)^{2/l}$ which, in applying the lemma will be raised to the power $l/2$.

Note that for all $i, j$ we have (very crudely in most cases)
\begin{equation*}
    |V_{i,j}^0|,|V_{i,j}^1|\le |V''_i|\Delta^{l-1}< \epsilon n^2p^2\Delta^{l-1}.
\end{equation*}
We apply Lemma \ref{basic} with
\begin{align*}
    m&= |V_{i,j}^0||V_{i,j+1}^1| < \epsilon^2n^4p^4\Delta^{2l-2},\\
    \alpha&= p, \text{ and}\\
    K&=(mp)^{-1} \left(\frac{\delta}{4l}\right)^{2/l}n^2p^2.
\end{align*}
A little checking (using $p<n^{\frac{-5l}{5l+1}}$) confirms that, for example,
\begin{equation*}
K>n^{1/6l}.
\end{equation*}
Thus for specified $i, V_i'',$ and $j$ Lemma \ref{basic} gives
\begin{equation}\label{72nablaprob}
\Pr\left(|\nabla(V_{i,j}^0,V_{i,j+1}^1)|> \left(\frac{\delta}{4l}\right)^{2/l}n^2p^2\right)<
\exp\left[- \frac{(\delta/(4l))^{2/l} n^2p^2t}{6l}\right],
\end{equation}
and summing over possibilities for $i, V_i'',$ and $j$ (recalling that we have $|V_i''|<\epsilon n^2p^2$) gives (\ref{nablabound}):
\begin{align*}
    \Pr&\left(\exists i,j \text{ with }|\nabla(V_{i,j}^0,V_{i,j+1}^1)|> \left(\frac{\delta}{4l}\right)^{2/l}n^2p^2\right)\\
    &<l^2 \sum_{w<\epsilon n^2p^2}{n \choose w}\exp\left[- \frac{(\delta/(4l))^{2/l} n^2p^2t}{6l}\right]\\
    &= \exp[-\Omega(n^2p^2t)].
\end{align*}
Here for the final bound we use that $w\ln(en/w)< \epsilon n^2p^2t$ and $\epsilon$ is small enough (see (\ref{epsilon})).

To apply Lemma \ref{shearer} here let $\F$ be the hypergraph on $V=V(\Hb)$ where each edge is the vertex set of a cycle using only vertices in $\bigcup_{j=1}^l V_{i,j}$. Again let $\HH$ be the hypergraph on $V$ with edges $\{H_j \coloneqq V_j \cup V_{j+1}\}_{j \in [l]}$. Thus each vertex belongs to exactly two edges of $\HH$.
Furthermore, (\ref{nablabound}) says
\begin{equation*}
|\Tr(\F, H_j)|\le |\nabla(V_{i,j},V_{i,j+1})|\le |\nabla(V_{i,j}^0,V_{i,j+1}^1)|<^*\left(\frac{\delta}{4l}\right)^{2/l}n^2p^2. 
\end{equation*}
Thus Lemma \ref{shearer} gives
\begin{equation*}
    |\F|\le \prod_{H \in \HH}|\Tr(\F,H)|^{1/2}<^*\left(\left(\frac{\delta}{4l}\right)^{2/l}n^2p^2\right)^{l/2}<\left(\frac{\delta}{4l}\right) n^lp^l,
\end{equation*}
as desired.
So, summing over choices for $i$, there are less than $(\delta/4) n^lp^l$ cycles using some $v_i \in V''_i$, as desired. 
\end{proof}

\printbibliography
\end{document}